\documentclass[11pt,oneside,a4paper]{article}
\usepackage[top=20mm, right=20mm, bottom=20mm, left=20mm]{geometry}
\usepackage{amsfonts, amsbsy, amsmath, amsthm, amssymb, latexsym, comment}
\usepackage{mathrsfs}
\usepackage{color}
\usepackage{lscape}

\newtheorem{theorem}{Theorem}[section]
\newtheorem{lemma}[theorem]{Lemma}

\newtheorem{conj}[theorem]{Conjecture}

\theoremstyle{definition}
\newtheorem{remark}[theorem]{Remark}

\newtheorem*{lemma*}{Lemma}

\newcommand{\C}{\mathscr{C}}
\newcommand{\Chat}{\widehat{\mathscr{C}}}

\title{Centres of blocks of finite groups with trivial intersection Sylow $p$-subgroups}
\author{Inga Schwabrow}

\begin{document}
\date{\today}
\maketitle

\vspace{-10mm}
\begin{center}

\small

\textit{FB Mathematik, TU Kaiserslautern, Postfach 3049, 67653 Kaiserslautern, Germany}

\textit{E-mail: inga.schwabrow@gmx.de}
\end{center}

\normalsize

\begin{abstract}\noindent
For finite groups $G$ with non-abelian, trivial intersection Sylow $p$-subgroups, the analysis of the Loewy structure of the centre of a block allows us to deduce that a stable equivalence of Morita type does not induce an algebra isomorphism between the centre of the principal block of $G$ and the centre of the Brauer correspondent. This was already known for the Suzuki groups; the result will be generalised to cover more groups with trivial intersection Sylow $p$-subgroups.

\end{abstract}

\textbf{Keywords}: blocks, trivial intersection defect groups, stable equivalence of Morita type

\textbf{AMS classification}: 20C05, 20C15

\section{Introduction}
Let $p$ be a fixed prime number. Let $G$ be a finite group whose order is divisible by $p$, and $P$ a Sylow $p$-subgroup of $G$.  
Throughout, $(K,\mathcal{O},k)$ is a $p$-modular system; in particular, $k$ is an algebraically closed field of characteristic $p$ and  $\mathcal{O}$ is a complete valuation ring whose residue field $k$ has characteristic $p$.

Local representation theory studies the connection between $kG$ and $kN_G(P)$;  Brou\'{e}'s abelian defect group conjecture (ADGC) predicts the existence of a derived equivalence between the principal blocks $B_0(kG)$ and $b_0(kN_G(P))$   if the Sylow $p$-subgroups of $G$ are abelian. We investigate the question of whether a derived equivalence can exist when the  Sylow $p$-subgroups are non-abelian, trivial intersection subgroups of $G$.

A group is said to have trivial intersection (TI) Sylow $p$-subgroups, if any two distinct Sylow $p$-subgroups intersect trivially. If a block $B$ of $G$ has TI defect groups $D$, then there exists a stable equivalence of Morita type between $B$ and its Brauer correspondent $b$ in $N_G(P)$ \cite[Section 11.2]{Konig}. This equivalence induces an algebra homomorphism between the respective stable centres.
Although $Z(B)$ and $Z(b)$ have the same dimension \cite[Theorem 9.2]{Blau}, and the stable centres are isomorphic \cite[Proposition 5.4]{BroueEquivalence}, surprisingly, there is in general no isomorphism between the two centres of $B$ and $b$.

The abelian defect group conjecture proposed by Brou\'{e} in $1990$ \cite[Question 6.2]{Broue1} provides a structural explanation for the close relationship between certain blocks. The conjecture claims that there is a derived equivalence between a block $B$  with abelian defect groups and its Brauer correspondent $b$ in $N_G(P)$.
An important consequence of  Brou\'{e}'s conjecture is that if a block has abelian defect groups then the derived equivalence  induces an isomorphism between the centres of the blocks $B$ and $b$ \cite[Theorem 1.5]{Broue1}.
In this article we will show the non-existence of such an isomorphism under the different assumption that the defect groups have the trivial intersection property, and are not abelian or of the form $p_{-}^{1+2}$.

\section{Preliminary}
Let $G$ be a finite group with trivial intersection Sylow $p$-subgroups $P$. Since we are interested in comparing the principal block algebras of $G$ and $N_G(P)$, we may assume that $P$ is not normal in $G$. Since $O_p(G)$ is the intersection of all Sylow $p$-subgroups of $G$, we have $O_p(G)=1$. The following remark allows us to further assume that $O_{p'}(G)=1$.

\begin{remark}\label{rema:Reduction} Let $G$ be a finite group and $H$  a normal $p'$-subgroup of $G$. By the Fong-Reynolds reduction (\cite{FongCharacters}, \cite{ReynoldsNormal}), we have  $B_0(G)$ is Morita equivalent to $B_0(G/H)$, denoted $B_0(G)\sim_{M}B_0(G/H)$.
Hence 
\[\begin{array}{rcll}
B_0(N_G(P))&\sim_{M} & B_0\left(    \frac{N_G(P)}{N_G(P)\cap O_{p'}(G)}\right)& \text{ by F-R reduction}\\[10pt]
&\cong &  B_0\left(    \frac{N_G(P)O_{p'}(G) }{O_{p'}(G)  }\right)& \text{ by  2nd Isomorphism Theorem}\\[10pt]
&=& B_0(N_{\overline{G}}(\overline{P}))& \text{ by \cite[Result 3.2.8]{Kurzweil}},
\end{array} \]
and  $B_0(G)\sim_{M} B_0(G/O_{p'}(G))$. Since Morita equivalent blocks have isomorphic centres \cite[Corollary 3.5]{BroueEquivalence}, $Z(B_0(\overline{G}))$ is  isomorphic to $Z(B_0(N_{\overline{G}}(\overline{P}))$ if and only if  $Z(B_0(G))$ is isomorphic to $Z(B_0(N_G(P)))$.
\end{remark}

\vspace{4mm}\noindent
The $p$-local rank, $plr(G)$, of $G$ is defined to be the length of a longest chain  in the set of radical $p$-chains of $G$ (for details see \cite{EatonPlocalRank} or \cite{RobinsonLocalStruc}); $plr(G)=0$ if and only if $G$ has a normal Sylow $p$-subgroup.   Moreover, if $plr(G)>0$, then $plr(G)=1$ if and only if $G/O_p(G)$ has TI Sylow $p$-subgroups \cite[Lemma 7.1]{RobinsonLocalStruc}. Hence if $G$ has TI Sylow $p$-subgroups, then $plr(G)=1$. In particular, the following lemma can be applied to $G$.

\begin{lemma}\label{lm:EatonReduction}\cite[Lemma 2.4]{EatonPlocalRank} Let $G$ be a finite group with $plr(G)=1$ and $O_p(G)=O_{p'}(G)=1$. Then there is a unique non-trivial minimal normal subgroup $S$ of $G$. Furthermore, $S$ is non-abelian simple, $plr(S)=1$ and $G$ is isomorphic to a subgroup of the automorphism group of $S$.
\end{lemma}

\begin{lemma}\label{lm:ClassificationTI}\cite[Lemma 3.2]{EatonPlocalRank} Let $p$ be a prime and $S$ be a non-abelian, simple group with $plr(S)=1.$ Then $(p,S)$ is one of the following:
\begin{description}
  \item[(a)] $(2, {^2B_2}(2^{2m+1}))$, $m\geq1$;
  \item[(b)] $(3, {^2G_2}(3^{2m+1}))$, $(3,PSL(3,4))$, $(3, {^2G_2}(3)')$, $(3,M_{11})$,  $m\geq1$;
  \item[(c)] $(5, {^2B_2}(2^5))$, $(5, {^2F_4}(2)')$, $(5,McL)$;
  \item[(d)] $(11,J_4)$;
  \item[(e)] $(p,PSL(2,p^m))$, $(p,PSU(3,p^m))$, $m\geq1$.
\end{description}
\end{lemma}
The following lemma appears in \cite{EatonPlocalRank}. However the proof presented there is incomplete as the group $PSL(3,4)$ was not covered; we include a short proof here for completeness of the statement.

\begin{lemma}\label{lm:gcd1} Let $G$ and $S$ be as in Lemma \ref{lm:EatonReduction}. Then gcd$(p,[G:S])=1$ except when $(p,S)=(3, {^2G_2}(3)')$ or $(5, {^2B_2}(2^5))$.
\end{lemma}
\begin{proof}
Suppose $(p,S)\neq  (3,PSL(3,4))$. Then the result follows by \cite[Lemma 3.3]{EatonPlocalRank}.

Suppose $S\cong PSL(3,4)$ and $S\lhd G$ such that $[G:S]=3$.  Let $P\in {\rm Syl}_3(G)$ and take an element $x\in P,x\not\in S$.  By \cite{ConwayAtlasFiniteGroups}, $|C_G(x)|$ is divisible by $5$ or $7$.

Note that since $\frac{N_G(P\cap S)}{N_S(P\cap S)}=\frac{N_G(P\cap S)}{N_G(P\cap S)\cap S}\cong \frac{N_G(P\cap S)S}{S}\leq\frac{G}{S}$, we must have that $[N_G(P\cap S):N_S(P\cap S)]$ divides $3$. Moreover, $N_S(P\cap S) \cong (P\cap S)\rtimes Q_8$  \cite{ConwayAtlasFiniteGroups}, which implies that $|N_G(P\cap S)|$ divides $3^3\cdot 2^3$. Finally since $N_G(P)\leq N_G(P\cap S)$, then gcd$(5,|N_G(P)|)=1=$ gcd$(7,|N_G(P)|)$. If $P$ is trivial intersection, then $C_G(x)\leq N_G(P)$, leading to a contradiction. Hence if $G$ has trivial intersection Sylow $p$-subgroups, then $G$ is a $p'$-extension of $S$.
\end{proof}

\noindent
Motivated by the examples presented in the next section, we make the following conjecture.
Let $G$ be a finite group and  $B=B_0(G)$ with TI defect groups $D$. Let $b\in Bl(N_G(D))$ such that $b^G=B$. Then
\begin{description}
  \item[for $p=2$,]  $B$ is derived equivalent to $b$ if and only if $D$ is abelian or generalised quaternion;
  \item[for $p>2$,]  $B$ is derived equivalent to $b$ if and only if $D$ is abelian or $D\cong 3_{-}^{1+2}$, $5_{-}^{1+2}$.
\end{description}

The aim of this article is to prove that the centres of $B$ and $b$ are not isomorphic when  $D$ is ``small'' and not abelian; we do this for  characteristic $p=3$ and $5$.
The main obstacle in proving the conjecture in full generality comes from the projective special unitary groups. A small number of individual cases are considered in Section \ref{sec:cases}, however further research extending the results presented is required to establish the non-existence of an isomorphism for $PSU(3, p^n)$ for an arbitrary prime $p$ and integer $n$. In particular, our bound on the sizes of the defect groups considered arises from this constraint. In addition, the question of what happens in the automorphism groups of the Suzuki and Ree groups is still open.

\section{Explicit calculations}\label{sec:cases}

The Loewy length of an algebra $A$, denoted $LL(A)$, is defined to be the nilpotency length of its Jacobson radical $J(A)$. Calculating the Loewy length of the centre of a block, and in particular the dimensions of the radical layers, form an important tool in establishing the non-existence of an isomorphism between two centres of blocks.  

In \cite{Cliff}, Cliff proved that for $G=$ $^2B_2(q)$, where $q=2^{2m+1}\geq8$, the centre of the principal $2$-block of $kG$ and the centre of the Brauer correspondent in $kN_G(P)$ are isomorphic over a field of characteristic $2$, but not over a discrete valuation ring $\mathcal{O}$. In particular, Cliff  concluded that an isomorphism in characteristic $2$ exists from the fact that the centres have the same dimension over $k$ and the Jacobson radical squared of each centre is equal to zero.

In \cite{BroughS},  the authors prove that in characteristic $3$, the centre of the principal block of the Ree group, $^2G_2(q)$, has Loewy length $3$ while the centre of its Brauer correspondent has Loewy length $2$.  It therefore follows that $Z(B_0)\not\cong Z(b_0)$.
In this section the difference in the dimension of the radical squared for various groups allows us to draw the same conclusion.

Due to the sizes of the groups considered, the computations were carried out in the computer algebra system GAP \cite{GAP}; the code can be found in the Appendix. 

\begin{remark} 
All groups considered in this article have trivial intersection Sylow $p$-subgroups.
Hence as a consequence of Green's Theorem \cite[Theorem 3]{Green68}, the group algebra decomposes as blocks of full defect, with defect groups $P$, and blocks of defect zero. Moreover, the number of blocks of $kG$ with defect groups $P$ is equal to the number of $p$-regular conjugacy classes with defect groups $P$ \cite{BrauerNesbitt}. It can be checked from the character table, that for the groups where we explicitly calculate the dimensions of the Loewy layers, there is only one block of full defect, the principal block. Additionally, the Jacobson radical of the centre of blocks with defect zero is zero; hence ${\rm dim}(J^n(Z(kG))={\rm dim}(J^n(Z(B_0))$ for all $n\geq 1$. Consequently, the calculations in GAP can be done over the centre of the group algebra, without having to worry about the principal block idempotent. 

Finally, for the groups $G$ considered in this article, $N_G(P)$ is $p$-solvable  and $O_{p'}(G)=1$. Hence the group algebras $kN_G(P)$ are indecomposable \cite[Proposition III.1.12]{KarpilovskyJacRadical} and $b_0=kN_G(P)$.
\end{remark}

\subsection{The McLaughlin group $McL$, and ${\rm Aut}(McL)$, with $p=5$}\label{sec:McLp5}
Consider the McLaughlin group and  fix the prime $p=5$.
\subsubsection{The group $McL$ }
Let $P$ be a Sylow $5$-subgroup of $McL$; note that $P$ is not abelian. The normaliser $N_G(P)\cong (((C_5\rtimes C_5)\rtimes C_5)  \rtimes C_3)\rtimes C_8$ splits into $19$ conjugacy classes \cite{ConwayAtlasFiniteGroups}; hence ${\rm dim}(Z(kN_G(P)))$ $=19$. 

For $g\in G$, let $\C(g)=\{h^{-1}gh\;|\; h\in G\}$ denote the conjugacy class of $g$. Then the conjugacy class sum of $g$ is defined to be $\Chat(g)= \sum_{h\in \C(g)} h$.

From \cite{ConwayAtlasFiniteGroups}, all non-trivial conjugacy classes have class size divisible by $5$ except $\C(5A)$ which has class size $|\C(5A)|=4$; hence a basis for $J(Z(kN_G(P)))$ is given by
\[ \mathfrak{B}_{N_G(P)}=\{ \Chat(x)\;|\; x\in  \mathscr{P}, x\neq 1_{N_G(P)}, x\not\in\C(5A) \} \cup \{\Chat(5A)+1\}.\]
There exist basis elements $b,b'\in\mathfrak{B}_{N_G(P)} $  such that
$b\cdot b'  \neq 0$; in particular the following distinct non-zero multiplications occur
\[\begin{array}{rcl}
\Chat(3A)\cdot \Chat(3A)&=&\Chat(3A)+\Chat(15A)+\Chat(15B);\\
\Chat(5B)\cdot \Chat(10A)&=&\Chat(2A)+\Chat(10A);\\
\Chat(2A)\cdot \Chat(3A)&=&\Chat(6A)+\Chat(30A)+\Chat(30B).\\
\end{array}\]
Any other pair of conjugacy class sums in $\mathfrak{B}_{N_G(P)}$ either multiplies to zero in $kN_G(P)$ or is a non-zero multiple of the three given multiplications.
Hence $J^2(Z(kN_G(P)))$ has dimension $3$ and a basis given by \[\mathfrak{B}^{(2)}_{N_G(P)}=\{\Chat(3A)+\Chat(15A)+ \Chat(15B),\Chat(2A) +\Chat(10A), \Chat(6A) + \Chat(30A) + \Chat(30B)\}.\]
Next we need to establish whether $J^3(Z(kN_G(P)))=0$. Due to the size of the group, we use GAP to explicitly calculate that for all $b,b',b'' \in \mathfrak{B}_{N_G(P)}$ we have $b\cdot b'\cdot b''=0$ (the code can be found in the Appendix). Hence  $J^3(Z(kN_G(P)))=0$ and so $LL(Z(kN_G(P)))=3$.

In characteristic $5$, the group algebra of the McLaughlin group decomposes into $6$ blocks: the principal block, $B_0$, of defect $3$ and five blocks of defect zero. By a result of Blau and Michler  \cite[Theorem 9.2]{Blau},  ${\rm dim}(Z(B_0))=19$.

All non-trivial conjugacy classes of $McL$ have class size divisible by $5$, except for $\C(5A)$ which has size $|\C(5A)| = 1197504$. Hence consider the set
\[ \mathfrak{D}_G=\{ \Chat(x)e_0\;|\; x\in \mathscr{P}, x\neq 1_G, x\not\in  \C(5A) \} \cup \{(\Chat(5A)+1)e_0\}.\]
This set is not linearly independent, however it is a spanning set for $J(Z(kGe_0))$, which is enough for our calculations.
In GAP  we can calculate that for all $b,b',b''\in\mathfrak{D}_G $ we have $b\cdot b'\cdot b''=0$.

\noindent
At the same time note that there exist elements $b,b'\in\mathfrak{D}_G$  such that $b\cdot b' \neq 0$.
More precisely, we calculate in GAP that ${\rm dim}(J^2(Z(kGe_0)))={\rm dim}(J^2(Z(kG)))=4$. The calculations above lead to the following theorem.
\begin{theorem}\label{th:DimMcL}
Let $G=McL$, $P\in{\rm Syl}_5(P)$ and $k$ an algebraically closed field of characteristic $p=5$. Then $LL(Z(kGe_0))= LL(Z(kN_G(P))) =3$. Moreover,
\[ {\rm dim}(J^2(Z(kGe_0)))= 4 \neq 3 = {\rm dim}(J^2(Z(kN_G(P)))), \]
and therefore $Z(kGe_0) \not\cong Z(k)$.
\end{theorem}

\subsubsection{The group $Aut(McL)$}
Let $G={\rm Aut}(McL) \cong McL.2$, which has $33$ conjugacy classes. The group algebra $kG$ decomposes into $7$ blocks: the principal block $B_0$ of defect $3$ and $6$ blocks of defect zero. As usual, $b_0$ is simply given by the group algebra $kN_G(P)$.

Using the same methods  as those  for $McL$, the following results are obtained:
\[\begin{array}{cccccc}
\hline
\text{Block }B& \text{defect}& {\rm dim}(Z(B))& LL(Z(B))&{\rm dim}(J^2(Z(B)))\\
\hline
B_0&  3& 26& 3 & 5 \\
b_0&  3&   26&  3 &4  \\
\hline
\end{array}\]
Similarly to the result in the McLaughlin group, we cannot have an isomorphism of the centres.
\begin{theorem}\label{th:DimAutMcL}
Let $G={\rm Aut}(McL)$, $P\in{\rm Syl}_5(P)$ and $k$ a field of characteristic $p=5$. Then
\[{\rm dim}(J^2(Z(kGe_0)))= 5 \neq 4 = {\rm dim}(J^2(Z(kN_G(P)))), \]
and therefore  $Z(kGe_0) \not\cong Z(kN_G(P))$.
\end{theorem}

\subsection{The Janko group $J_4$ with $p=11$}\label{Sec:Janko4}
Let $G=J_4$, a sporadic simple group, and  fix the prime $p=11$. In characteristic $11$, the group algebra $kJ_4$ decomposes into $14$ blocks: the principal block, $B_0$, of defect $3$ and $13$ blocks of defect zero.
Let $P$ be a Sylow $11$-subgroup of $J_4$. The normaliser $N_G(P)$ has  $49$ conjugacy classes \cite{ConwayAtlasFiniteGroups}; hence ${\rm dim}(Z(kN))={\rm dim}(Z(kGe_0)) =49$.

Using GAP, the following results are obtained:
\[\begin{array}{cccccc}
\hline
\text{Block }B& \text{defect}& {\rm dim}(Z(B))& LL(Z(B))&{\rm dim}(J^2(Z(B)))\\
\hline
B_0&  3& 49& 3 & 5 \\
b_0&  3 &49&  3 &4  \\
\hline
\end{array}\]
We get the following  theorem.

\begin{theorem}\label{th:DimJ4}
Let $G=J_4$, $P\in{\rm Syl}_{11}(P)$ and $k$ an algebraically closed field of characteristic $p=11$. Then $LL(Z(kGe_0))= LL(Z(kN_G(P)))=3$. Moreover,
\[ {\rm dim}(J^2(Z(kGe_0)))= 5 \neq 4 = {\rm dim}(J^2(Z(kN_G(P)))),\]
and  therefore $Z(kGe_0) \not\cong Z(kN_G(P))$.
\end{theorem}

\subsection{The projective special unitary groups}\label{sec:PSUgroups}
In this section we consider some projective special unitary groups and calculate some block theoretic properties which are required for consideration of our question. 
The calculations follow the same idea as the method given for $McL$ and $J_4$; the GAP code used in the calculations can be found in the Appendix. We only summarise the results here; for more details see \cite{SchwabrowPhD}.

\begin{table}[h!]
\centering
\caption{Summary of block information:}\label{tb:PSUblocks}
\[\begin{array}{lc|ccccccccc}
\hline
\text{G} &\text{ prime}&\text{Block }B& \text{defect}& {\rm dim}(Z(B))& LL(Z(B))&{\rm dim}(J^2(Z(B)))\\
\hline
PSU(3,4)&2 & B_0& 6  &21 & 3&5  \\
& & b_0&  6 &21&  3&4  \\
\hline
PSU(3,8)&2 & B_0& 9  &27 & 3&3  \\
& & b_0&  9 &27&  3&2  \\
\hline
PSU(3,3)&3 & B_0& 3  &13 & 3&4  \\
& & b_0&  3 &13&  3&3  \\
\hline
PSU(3,3):2&3 & B_0& 3 &14 & 3&4  \\
& & b_0&  3 &14&  3&3  \\
\hline
PSU(3,9)&3 & B_0& 6  &91 & 3&10  \\
& & b_0&  6&91&  3&9 \\
\hline
PSU(3,9):2&3 & B_0& 6  &62 & 3&7  \\
& & b_0&  6&62&  3&6 \\
\hline
PSU(3,9):4&3 & B_0& 6  &46 & 3&5  \\
& & b_0&  6&46&  3&4 \\
\hline
PSU(3,5)&5 & B_0& 3  &13 & 3&2  \\
& & b_0&  3 &13&  3&1  \\
\hline
PSU(3,5):2&5 & B_0& 3  &17 & 3&3  \\
& & b_0&  3 &17&  3&2  \\
\hline
PSU(3,5):3&5 & B_0& 3  &31 & 3&6  \\
& & b_0&  3 &31&  3&5  \\
\hline
PSU(3,5):S_3&5 & B_0& 3 &26 & 3&5  \\
& & b_0&  3 &26&  3&4  \\
\hline
PSU(3,7)&7 & B_0& 3  &57 & 3&8  \\
& & b_0&  3 &57&  3&7  \\
\hline
PSU(3,7):2&7 & B_0& 3  &42 & 3&6  \\
& & b_0&  3 &42&  3&5  \\
\hline
\end{array}\]
\end{table}

\begin{theorem}\label{th:PSU}
Let $G$ be any of the group given in Table \ref{tb:PSUblocks} with corresponding prime $p$. Let $P\in {\rm Syl}_p(G)$ and $k$ is an algebraically closed field of characteristic $p$. Then
\[ {\rm dim}(J^2(Z(kGe_0)))  = {\rm dim}(J^2(Z(kN_G(P))))\;+1, \]
and so $Z(kGe_0) \not\cong Z(kN_G(P))$.
\end{theorem}

Some of the calculations regarding $PSU(3,p^r)$ were done independently by Bouc and Zimmermann in a recent paper \cite{BoucZi}. Motivated by a question of Rickard, the authors state the same results for the principal $p$-block of the group $PSU(3,p^r)$ and its Brauer correspondent for $p^r\in \{3,4,5,7,8\}$.
Moreover, Bouc and Zimmermann explicitly calculate ${\rm dim}(J^2(Z(b_0)))$  and the Loewy length for the normaliser \cite[Theorem 41]{BoucZi}. 
They also make the observation that the examples give rise to the following conjecture.
\begin{conj}\cite[Remark 15]{BoucZi} Let $G=PSU(3,p^r)$, $B_0$ the principal block of $kG$, and $b_0$ the Brauer correspondent of $B_0$ in $kN_G(P)$. Then
\[{\rm dim}(J^2(Z(B_0)))= 1+  {\rm dim}(J^2(Z(b_0))).\]
\end{conj}
For us, the conjecture would have the following consequence.
\begin{conj} Let $G=PSU(3,p^r)$, $B_0$ the principal block of $kG$, and $b_0$ the Brauer correspondent of $B_0$ in $kN_G(P)$. Then
\[Z(B_0)\not\cong Z(b_0). \]
\end{conj}


\section{Main Theorems}
\begin{remark}\label{rema:DerivedPerfectIso} Two blocks $B$ and $b$ being derived equivalent implies the  existence of a perfect isometry from ${\rm Irr}(B)$ to ${\rm Irr}(b)$ \cite[Theorem 3.1]{Broue1}, which in turn induces an algebra isomorphism between $Z(B)$ and $Z(b)$ \cite[Theorem 1.5]{Broue1}. Hence if $Z(B)\not \cong Z(b)$ then no such perfect isometry can exist. On the other hand, Brou\'{e}'s abelian defect group conjecture (ADGC) states that if $B$ has abelian defect groups then $B$ and its Brauer correspondent $b$ are derived equivalent, and hence there exists an isomorphism between the centres of the two blocks.
\end{remark}


\begin{theorem}\label{th:Star} Fix $p=3$. Let $G$ be a finite group and $B_0\in Bl(G)$ be the principal block with TI defect groups $D$ such that $|D| \leq 3^8 $; let $b_0\in Bl(N_G(D))$. Then there exists an isomorphism between $Z(B_0)$ and $Z(b_0)$ if and only if $D$ is abelian or $D\cong 3_{-}^{1+2}$.
\end{theorem}
\begin{proof} We apply Lemma \ref{lm:EatonReduction} and Remark \ref{rema:DerivedPerfectIso}, and  individually consider the cases given in Lemma \ref{lm:ClassificationTI}  which relate to $p=3$. The structure descriptions of $S$ and ${\rm Aut}(S)$  given below  follow from \cite{ConwayAtlasFiniteGroups}.

\underline{$S=$ $^2G_2(3^{2m+1})$}\newline
The smallest simple group is $S=$ $^2G_2(3^{3})$ and the defect group of the principal block has size $|D|=3^9$. Hence this case is  excluded in the statement. 
However, the reader should note that the result is true for $^2G_2(3^{2m+1})$ for all $m\geq1$ \cite{BroughS}. The question of what happens in the automorphism group remains an open question.

\underline{$S=PSL_3(4)$ }\newline
Suppose $S\leq G\leq {\rm Aut}(S)$. Then, by Lemma \ref{lm:gcd1}, $p\;\nmid\;[G:S]$ and $D\cong C_3\times C_3$. The ADGC holds in this case \cite{KoshitaniKunugi2002}.

\underline{$S=$ $^2G_2(3)'$} \newline
Note that $|Out(S)|=3$ so $G\cong S$ or $G\cong {\rm Aut}(S)$.
If $G\cong S$  then $D$ is cyclic and the ADGC holds in this case (\cite{Linck1}, \cite{Rick1}, \cite{Rouq1}).
If $G\cong {\rm Aut}(S)={^2G_2}(3)$, then  $D\cong 3_{-}^{1+2}$ and  there exists a perfect isometry between $B_0$ and $b_0$ \cite[Example 4.3]{HollKoshKun}; it is not known if the two blocks are derived equivalent.

\underline{$S= M _{11}$}\newline
Since $S$ has trivial outer automorphism group, $G\cong S$.  The principal $3$-block of $M_{11}$ has abelian defect group $D\cong C_3\times C_3$, and the ADGC has been verified in this case \cite{Okuyama1}.

\underline{$S=PSL(2,3^m)$}\newline
If $G\cong PSL_2(3^m)$ where $1\leq m\leq 5$, then $D \cong (C_3)^m$ is abelian and the ADGC has been verified \cite{Okuyama2}.
If $G$ is such that $S<G\leq {\rm Aut}(S)$, then by Lemma \ref{lm:gcd1}, gcd$(p,[G:S])=1$. Hence by \cite{FongIsometries}, there exists a perfect isometry between $B_0$ and $b_0$.

\underline{$S=PSU(3,3^m)$}\newline
If $S=PSU(3,3)$ or $PSU(3,9)$,  and
$S\leq G\leq {\rm Aut}(S)$, then  $D$ is not abelian and by Theorem  \ref{th:PSU}, $Z(B_0)\not\cong Z(b_0)$. If $S=PSU(3,3^m)$ for $m>2$, then $|D| >3^8$.
\end{proof}

\vspace{6mm}\noindent
We next consider the principal $5$-blocks. The group $PSU(3,25)$ has a principal block with defect group $D$ such that $|D|=5^6$. Hence we restrict to blocks with defect groups of smaller sizes.
\begin{theorem}\label{th:MainB} Fix $p=5$. Let $G$ be a finite group and $B_0\in Bl(G)$ be the principal block with TI defect groups $D$ such that $|D| \leq 5^5 $; let $b_0\in Bl(N_G(D))$. Then there exists an isomorphism between $Z(B_0)$ and $Z(b_0)$ if and only if $D$ is abelian or $D\cong 5_{-}^{1+2}$.
\end{theorem}
\begin{proof} As in Theorem \ref{th:Star}, we apply Lemma \ref{lm:EatonReduction} and
Remark \ref{rema:DerivedPerfectIso}, and  individually consider the cases given in Lemma \ref{lm:ClassificationTI}  which relate to $p=5$. The structure descriptions of $S$ and ${\rm Aut}(S)$  given below follow  from \cite{ConwayAtlasFiniteGroups}.


\underline{$S=$ $^2B_2(2^5)$}\newline
Note that $|Out(S)|=5$ so $G\cong S$ or $G\cong {\rm Aut}(S)$.
If $G\cong S$  then $D$ is cyclic and the ADGC holds in this case (\cite{Linck1}, \cite{Rick1}, \cite{Rouq1}).
If $G\cong {\rm Aut}(S)$, then $D\cong 5_{-}^{1+2}$ and there exists a perfect isometry between $B_0$ and $b_0$ \cite[Example 4.4]{HollKoshKun}; it is not known if the two blocks are derived equivalent.

\underline{$S=$ $^2F_4(2)'$}\newline
Note that $|Out(S)|=2$ so $G\cong S={^2F_4}(2)'$ or $G\cong {\rm Aut}(S)={^2F_4}(2)$.
In either case, $D\cong C_5\times C_5$ is abelian and the ADGC has been verified by Robbins \cite{Robbins}.

\underline{$S=McL$}\newline
Let $S\leq G\leq {\rm Aut}(S)$. Then the defect groups of $B_0(G)$  are not abelian and by Theorem  \ref{th:DimMcL} and Theorem \ref{th:DimAutMcL}, $Z(B_0)\not\cong Z(b_0)$.

\underline{$S=PSL(2,5^m)$}\newline
If $G \cong PSL_2(5^m)$ where $1\leq m\leq 5$, then $D \cong (C_5)^m$ is abelian and the ADGC has been verified \cite{Okuyama2}.
If $G$ is such that $S<G\leq {\rm Aut}(S)$ then, by Lemma \ref{lm:gcd1}, gcd$(p,[G:S])=1$. Hence by \cite{FongIsometries}, there exists a perfect isometry between $B_0$ and $b_0$.

\underline{$S=PSU(3,5^m)$}\newline
If $PSU(3,5)=S\leq G\leq {\rm Aut}(S)\cong PSU(3,5):S_3$, then $D \cong (C_5 \rtimes C_5)\rtimes C_5$ is not abelian and by Theorem \ref{th:PSU}, $Z(B_0)\not\cong Z(b_0)$. If $S=PSU(3,5^m)$ for $m>1$, then $|D| >5^5$.
\end{proof}

\begin{remark} The two exceptions $3_{-}^{1+2}$ and $5_{-}^{1+2}$ arise from a weak conjecture of Brou\'{e} and Rouquier; this is discussed in  \cite[Conjectures 4.1]{HollKoshKun} and we restate it here to give a more complete picture of the context of our results. Let $B_0$ be the principal $p$-block of a finite group $G$ with a non-abelian Sylow $p$-subgroup $P$. Let $Q$ be the hyperfocal subgroup of $P$ in $G$, $Q=P\cap H$ where $H$ is the smallest normal subgroup of $G$ satisfying that $G/H$ is $p$-nilpotent. Rouquier conjectures that if $Q$ is abelian then the $p$-block $B_0$ and its Brauer correspondent $b_0$ in $N_G(Q)$ should be derived equivalent; the weaker version, as stated by Koshitani, Holloway and Kunugi, conjectures the existence of a perfect isometry in this case.

In most of our cases of non-abelian, trivial intersection defect groups, we have $H=G$ and so $Q=P\cap H=P$; therefore $Q$ is not abelian, and the weaker conjecture does not apply. The only exceptions are precisely the examples $G={\rm Aut}(^2G_2(3)')$ and $G={\rm Aut}(^2B_2(2^5))$, as discussed in Examples $4.3$ and $4.4$ of \cite{HollKoshKun}.
\end{remark}

\section*{Acknowledgments}
The work formed part of the author's PhD research, which was supported by EPSRC grant $1240275$. The author would further like to thank Prof. Gunter Malle for providing the opportunity to write up this research.

\section*{Appendix}\label{AppendixGAP}
We display here the GAP \cite{GAP}  code used in the computations. The author would like to thank Benjamin Sambale for kindly providing this short and efficient code. The author originally used a much more cumbersome code which can be found in \cite{SchwabrowPhD}.

\begin{verbatim}
CenterOfGroupAlgebra:=function(G,p)
local ct,dim,i,j,l,k,SCT;
if IsCharacterTable(G) then ct:=G; else ct:=CharacterTable(G); Irr(ct); fi;
dim:=NrConjugacyClasses(ct);
SCT:=EmptySCTable(dim, Zero(GF(p)), "symmetric");
for i in [1..dim] do
     for j in [i..dim] do
         l:=[];
         for k in [1..2*dim] do
             if k mod 2=1 then
l[k]:=ClassMultiplicationCoefficient(ct,i,j,(k+1)/2)*One(GF(p)); else
l[k]:=k/2; fi;
         od;
         SetEntrySCTable(SCT,i,j,l);
     od;
od;
return AlgebraByStructureConstants(GF(p),SCT);
end;
\end{verbatim}
The function defined can now be used to calculate the Loewy layer of the centre of the group algebra. For larger groups, it is more efficient to load $G$ directly as the corresponding character table, using the AtlasRep package of GAP and the command G:=CharacterTable("J4");.

\begin{verbatim}
#Define G and p
A:=CenterOfGroupAlgebra(G,p);            
J:=RadicalOfAlgebra(A);
JJ:=ProductSpace(J,J);
JJJ:=ProductSpace(J,JJ);
...
\end{verbatim}

\end{document}